\documentclass[12pt]{amsart}
\usepackage{graphicx}
\usepackage{latexsym}
\usepackage{fancyhdr}
\usepackage{amsmath, amssymb}
 \usepackage[utf8]{inputenc}
\usepackage[all]{xy}
\usepackage{pdflscape}
\usepackage{longtable}
\usepackage{rotating}
\usepackage{hyperref}
\usepackage{subfigure}
\usepackage{tensor}
\usepackage{enumerate}
\usepackage[textsize=small]{todonotes}

\usepackage{color}

\theoremstyle{plain}
\newtheorem{thm}{Theorem}[section]

\newtheorem{cor}[thm]{Corollary}
\newtheorem{lem}[thm]{Lemma} 
\newtheorem{prop}[thm]{Proposition}

\theoremstyle{definition}

\theoremstyle{remark}
\newtheorem{rem}[thm]{Remark}
\numberwithin{equation}{section}

\newtheorem{example}[thm]{Example}

\newcommand{\lgw}{\longrightarrow}
\newcommand{\lgm}{\longmapsto}
\newcommand{\si}{\sigma}

\newcommand{\Ker}{\text{Ker}}
\renewcommand{\Im}{\text{Im}}

\newcommand{\la}{\lambda}

\newcommand{\ma}{\mathfrak{m}}

\newcommand{\R}{\mathbb{R}}
\newcommand{\K}{\mathbb{K}}

\newcommand{\N}{\mathbb{N}}

\newcommand{\C}{\mathbb{C}}

\newcommand{\Q}{\mathbb{Q}}
\renewcommand{\t}{\tau}

\renewcommand{\lg}{\langle}
\newcommand{\rg}{\rangle}

\renewcommand{\phi}{\varphi}

\let\mathscr\mathcal

\DeclareMathOperator{\grad}{grad }

\newcommand{\prodscal}[2]{\left<#1,#2\right>}

\begin{document}
\title{Local topological algebraicity of analytic function germs}
\author[M. Bilski, A. Parusi\'nski, G. Rond]{Marcin Bilski, Adam Parusi\'nski, and Guillaume Rond}


\address{Department of Mathematics and Computer Science, Jagiellonian University, \L ojasiewicza 6, 30-348 Krak\'ow, Poland}
\email{marcin.bilski@im.uj.edu.pl}

\address{Laboratoire J. A. Dieudonn\'e, Universit\'e de Nice-Sophia Antipolis, Parc Valrose, 06108 Nice Cedex 02, France}
\email{adam.parusinski@unice.fr}

\address{Aix Marseille Universit\'e, CNRS, Centrale Marseille, I2M, UMR 7373, 13453 Marseille, France}
\email{guillaume.rond@univ-amu.fr}

\thanks{The authors were partially supported by ANR project STAAVF (ANR-2011 BS01 009)}

\date{January 11, 2014}


\begin{abstract}
T. Mostowski showed that every (real or complex) germ of
an analytic set is  homeomorphic to the germ of an algebraic set.
In this paper we show that every (real or complex) analytic function germ, defined
on a possibly singular analytic space,  is topologically equivalent to a 
polynomial function germ defined on an affine algebraic variety. 
\end{abstract}

\maketitle



\section{Introduction and statement of results}
The problem of approximation of analytic objects (sets or mappings) by algebraic ones has attracted many mathematicians,  see e.g. \cite{bochnakkucharz1984} and the bibliography therein.  
Nevertheless there are very few positive results if one requires that the approximation gives a  
homeomorphism between the approximated object 
and the  approximating one.  
In this paper we consider two cases of this problem: the local algebraicity of analytic sets and 
the local algebraicity of analytic functions.   
The problem can be considered over $\K = \R$ or $\C$.  

The local topological algebraicity of analytic sets has been established by Mostowski in \cite{mostowski}.  More precisely, given an analytic set germ $(V,0) \subset (\K^n,0)$, Mostowski shows the existence 
of a local homeomorphism $\tilde h: (\K^{2n+1},0) \to (\K^{2n+1},0)$ such that, after the embedding 
$(V,0) \subset (\K^n,0) \subset (\K^{2n+1},0)$, the image $\tilde h(V)$ is algebraic.  
It is easy to see that Mostowski's proof together with  Theorem 2 of \cite{bochnakkucharz1984} gives 
the following result.  

\begin{thm} \label{homeotoalgebraic}  
Let  $\K = \R$ or $\C$.  
Let  $(V,0) \subset (\K^n,0)$ be an analytic germ.  
 Then there is a homeomorphism $h: (\K^n,0) \to (\K^n,0)$ such that $h(V)$ 
 is the germ of an algebraic subset of $\K^n$.  
\end{thm}

Mostowski's Theorem seems not to be widely known.  Recently Fern\' andez de Bobadilla showed, by a method 
different from that of Mostowski, the local topological algebraicity of complex hypersurfaces with one-dimensional 
singular locus, see \cite{bobadilla}.  We remark that in \cite {mostowski} Mostowski states  his results 
only for $\K=\R$ but his proof works, word by word, for $\K=\C$.   

The first purpose of this paper is to present a short proof of Theorem \ref{homeotoalgebraic}.  
We follow closely  Mostowski's original approach that is based on two ideas, P{\l}oski's  version of Artin approximation, cf.   
\cite{ploski1974}, and Varchenko's theorem stating that the algebraic equisingularity of Zariski  implies topological 
equisingularity.   Our proof is shorter, but less elementary.  We use a corollary of Neron Desingularization, 
that we call the Nested Artin-P{\l}oski Approximation Theorem.  This result is shown in Section \ref{nestedploski} 
and the proof of Theorem \ref{homeotoalgebraic} is given in Section \ref{mostowskisection}.  
In Section \ref{varchenko} we recall Varchenko's results.

The second purpose of this paper is to show the local algebraicity of analytic functions.  
\begin{thm} \label{homeotopolynomial}  
Let  $\K = \R$ or $\C$.  
 Let  $g: (\K^n,0)\to (\K, 0)$ be an analytic function germ.  
 Then there is a homeomorphism $\sigma : (\K^n,0) \to (\K^n,0)$ such that $g\circ \sigma$ 
 is the germ of a polynomial.  
\end{thm}

The proof of Theorem \ref{homeotopolynomial}, presented in Section \ref{functiongerms}, is based on  
the Nested Artin-P{\l}oski Approximation Theorem and a refinement of Varchenko's method.

We end with the following generalization of Theorems \ref{homeotoalgebraic} and \ref{homeotopolynomial}.

\begin{thm} \label{generalhomeo}  
Let  $\K = \R$ or $\C$.  
Let  $(V_i,0) \subset (\K^n,0)$ be a finite family of analytic set germs and let  $g: (\K^n,0)\to (\K, 0)$ be an analytic function germ.  Then there is a homeomorphism $\sigma : (\K^n,0) \to (\K^n,0)$ such that $g\circ \sigma$ 
 is the germ of a polynomial, and  for each $i$, $\sigma^{-1} (V_i)$ 
 is the germ of an algebraic subset of $\K^n$.  
\end{thm}

\begin{cor}
Let  $g: (V,p)\to (\K,0)$ be an analytic function germ defined on the germ $(V,p)$ of an analytic space.  Then there 
exists an algebraic affine variety $V_1$, a point $p_1\in V_1$, the germ of a polynomial function $g_1:(V_1,p_1)\to 
(\K,0)$ and  a homeomorphism $\sigma : (V_1,p_1) \to (V,p)$ such that $g_1= g\circ \sigma$.  
\end{cor}

In Section \ref{examples} we present examples showing that Theorems \ref{homeotoalgebraic}, \ref{homeotopolynomial} and \ref{generalhomeo}  are false if we replace "homeomorphism" by "diffeomorphism".  
 We do not know whether these theorems hold true with "homeomorphism" replaced  
 by "bi-lipschitz homeomorphism". 

\begin{rem}
 We often identify the germ at the origin of a $\K$-analytic function $f:(\K^n,0) \to \K$ with its Taylor series that 
 is with a convergent power series.  We say that a ($\K$-)analytic function or a germ is \emph{Nash} if its graph is semi-algebraic.  Thus $f:(\K^n,0) \to \K$ is the germ of a Nash function if and only if its Taylor series 
is an algebraic power series.  A \emph{Nash set} is the zero set of a finitely many Nash functions.  
\end{rem}

\section{Nested Artin-P\l oski Approximation Theorem}\label{nestedploski}

We set $x=(x_1,...,x_n)$ and $y=(y_1,...,y_m)$. The ring of convergent power series  in $x_1$,..., $x_n$ is denoted by $\K \{ x\}$. If $A$ is a commutative ring then the ring of algebraic  power series with coefficients in $A$ is denoted by $A\lg x\rg$.

The following result is a corollary of Theorem 11.4 \cite{Sp} which itself is a corollary of N\'eron-Popescu desingularization (see \cite{Po}, \cite{Sp} or \cite{Sw} for the proof of this desingularization theorem in whole generality or \cite{Qu} for a proof in characteristic zero).  

\begin{thm}\label{nest_ploski}
Let $f(x,y)\in \K \lg x\rg[y]^p$ and let $y(x)\in\K \{x\}^m$ be a solution of $f(x,y)=0$. Let us assume that $y_i(x)$ depends only on $(x_1,..., x_{\si(i)})$ where $i\lgm \si(i)$ is an increasing function. Then there exist a new set of variables $z=(z_1,..., z_s)$, an increasing function $\t$, convergent power series $z_i(x)\in\K \{x\}$ vanishing at 0 such that $z_1(x)$,..., $z_{\t(i)}(x)$ depend only on $(x_1,..., x_{\si(i)})$, and  a vector of algebraic power series $y(x,z)\in\K \lg x,z\rg^m$ solution of $f(x,y)=0$ such that for every $i$, $y_i(x,z)\in\K \lg x_1,...,x_{\si(i)},z_1,...,z_{\t(i)}\rg$, and $y(x)=y(x,z(x))$.
\end{thm}

\begin{rem}
Theorem \ref{nest_ploski} remains valid if we replace  "convergent power series"  by  
"formal power series".
\end{rem}

\noindent For any $i$ we set:
$$A_i=\K \lg x_1,...,x_{i}\rg,$$
$$B_i=\K \{x_1,...,x_{i}\}.$$

\noindent We will need at several places the following two lemmas whose proofs are given later
(for the definition and properties of an excellent ring see 7.8 \cite{EGA2} or \cite{Ma}; a henselian local ring is a local ring satisfying the Implicit Function Theorem, see 18.5  \cite{EGA}).

\begin{lem} \label{lem} Let $B$ be an excellent  henselian local subring of $\K [[x_1,...,x_{i-1}]]$ containing $\K \lg x_1,...,x_{i-1}\rg$ and whose maximal ideal is generated by $x_1$,..., $x_{i-1}$. Then  the ring $A_i\otimes_{A_{i-1}}B$
is noetherian and its henselization is isomorphic to $B\lg x_i\rg$.
\end{lem}

\begin{lem}\label{lem2}
Let $B$ be an excellent henselian local subring of $\K [[x_1,...,x_{i-1}]]$ containing $\K \lg x_1,...,x_{i-1}\rg$ and whose maximal ideal is generated by $x_1$,..., $x_{i-1}$. Let $I$ be an ideal of $B[x_i]$. Then the henselization of $\frac{B[x_i]_{(x_1,...,x_i)}}{I}$ is isomorphic to $\frac{B\lg x_i\rg}{I}$. 
\end{lem}

\begin{proof}[Proof of Theorem \ref{nest_ploski}]
By replacing $f(x,y)$ by $f(x,y(0)+y)$ we may assume that $y(0)=0$.

For any $i$ let $l(i)$ be the largest integer such that $y_1(x)$,..., $y_{l(i)}(x)\in\K \{x_1,...,x_i\}$.\\
For any $i$ let $J_i$ be the kernel of the morphism
$$\phi_i : \K \lg x_1,...,x_i\rg[y_1,...,y_{l(i)}]\lgw \K \{x_1,...,x_i\}=B_i$$
defined by $\phi_i(g(x,y))=g(x,y(x))$. We define:
$$C_i= \frac{\K \lg x_1,...,x_i\rg[y_1,...,y_{l(i)}]}{J_i}.$$
Then $C_i$ is a finite type $A_i$-algebra and $C_i$ is a sub-$A_i$-algebra of $C_{i+1}$ since $J_{i}\subset J_{i+1}$.  The morphism $\phi_i$ induces a morphism $C_i\lgw B_i$ such that the following diagram is commutative:

$$\xymatrix{A_1 \ar[d] \ar[r] & A_2 \ar[d] \ar[r] & \cdots \ar[r] \ar[d] & A_n\ar[d]\\
C_1  \ar[d] \ar[r] & C_2 \ar[d] \ar[r] & \cdots  \ar[r] \ar[d] & C_n\ar[d]\\
B_1 \ar[r] & B_2  \ar[r] & \cdots \ar[r] & B_n}$$
By Theorem 11.4 \cite{Sp} (see also \cite{Te}) this diagram may be extended to a commutative diagram as follows
$$\xymatrix{A_1 \ar[d] \ar[r] & A_2 \ar[d] \ar[r] & \cdots \ar[r] \ar[d] & A_n\ar[d]\\
C_1  \ar[d] \ar[r] & C_2 \ar[d] \ar[r] & \cdots   \ar[d] \ar[r] & C_n\ar[d]\\
D_1 \ar[d] \ar[r] & D_2 \ar[d] \ar[r] & \cdots \ar[d]  \ar[r] & D_n\ar[d]\\
B_1 \ar[r] & B_2  \ar[r] & \cdots  \ar[r] & B_n}$$
where $D_1$ is a smooth $A_1$-algebra of finite type and $D_i$ is a smooth $D_{i-1}\otimes_{A_{i-1}}A_i$-algebra of finite type for all $i>1$. We will denote by $D'_{i-1}$ the ring $D_{i-1}\otimes_{A_{i-1}}A_i$ for all $i>1$ and set $D'_1=A_1$.

For any $i$ let us write $D_i=\frac{D'_{i-1}[u_{i,1},...,u_{i,q_i}]}{I_i}$. 
We may make a change of coordinates (of the form $u_{i,j}\longmapsto u_{i,j}+c_{i,j}$ for some $c_{i,j}\in \K$) in such way that the image of $u_{i,j}$ is in the maximal ideal of $B_i$ for any $i$ and $j$. Thus $D_i\lgw B_i$ factors through  the localization morphism $D_{i}\lgw (D_i)_{\ma_i}$ where $\ma_i=(x_1,...,x_i,u_{i,1},...,u_{i,q_i})$.   Let $D_i^h$ be the henselization of $(D_i)_{\ma_i}$. Since $B_i$ is a henselian local ring, the morphism $D_i\lgw B_i$ factors through $D_i^h$ by the universal property of the 
henselization. Still by this universal property the composition of the morphisms $D_{i-1}\lgw D_i\lgw D_i^h$ factors through $D_{i-1}^h$. Thus we have the following commutative diagram: 
$$\xymatrix{A_1 \ar[d] \ar[r] & A_2 \ar[d] \ar[r] & \cdots \ar[r] \ar[d] & A_n\ar[d]\\
C_1  \ar[d] \ar[r] & C_2 \ar[d] \ar[r] & \cdots   \ar[d] \ar[r] & C_n\ar[d]\\
D_1^h \ar[d] \ar[r] & D_2^h \ar[d] \ar[r] & \cdots \ar[d]  \ar[r] & D_n^h\ar[d]\\
B_1 \ar[r] & B_2  \ar[r] & \cdots  \ar[r] & B_n}$$
We will prove by induction that $D_i^h$ is isomorphic to $\K \lg x_1,...,x_i,z_1,...,z_{\la(i)}\rg$ where $i\lgw \la(i)$ is an increasing function and the $z_k$ are new indeterminates.

Since $D_1^h$ is the henselisation of $(D_1)_{\ma_1}=\frac{\K \lg x_1\rg[u_{1,1},...,u_{1,q_1}]}{I}_{(x_1,u_{1,1},...,u_{1,q_1})}$, $D_1^h$ is isomorphic to $\frac{\K \lg x_1,u_{1,1},...,u_{1,q_1}\rg}{I.\K \lg x_1,u_{1,1},...,u_{1,q_1}\rg}$ by Lemma \ref{lem2} and $D_1$ being smooth over $\K \lg x_1\rg$ means that the matrix $\left(\frac{\partial f_j}{\partial u_k}(0,0)\right)_{i,j}$, where the $f_j$ are generators of $I.\K \lg x_1,u_{1,1},...,u_{1,q_1}\rg$, has maximal rank (by the jacobian criterion for smoothness, see Proposition  22.6.7 (iii) \cite{EGA1}). Thus by the Implicit Function Theorem the ring $D_1^h$ is isomorphic to $\K \lg x_1,z_1,...,z_{\la(1)}\rg$ for some new indeterminates $z_1$,..., $z_{\la(1)}$. This proves the induction propery for $D_1^h$.

Now let us assume that the induction property  is true for $D_{i-1}^h$. By assumption $D_i$ is smooth over $D_{i-1}\otimes_{A_{i-1}}A_i$. Thus $D_i^h$ is smooth over the henselization of $D_{i-1}\otimes_{A_{i-1}}A_i$. 
By the universal property of the henselization the morphism from $D_{i-1}$ to the henselization of $D_{i-1}\otimes_{A_{i-1}}A_i$ factors through $D_{i-1}^h$ thus it factors through $D_{i-1}^h\otimes_{A_{i-1}}A_i$. Hence the henselization of $D_{i-1}\otimes_{A_{i-1}}A_i$ is isomorphic to the henselization of $D_{i-1}^h\otimes_{A_{i-1}}A_i$. 
But 
$$D_{i-1}^h\otimes_{A_{i-1}}A_i=\K \lg x_1,...,x_{i-1},z_1,...,z_{\la(i-1)}\rg\otimes_{\K \lg x_1,...,x_{i-1}\rg}\K \lg x_1,...,x_i\rg.$$
Its henselization is isomorphic to $\K \lg x_1,...,x_{i},z_1,...,z_{\la(i-1)}\rg$ by Lemma \ref{lem}.
 This shows that $D_i^h$ is smooth over $\K \lg x_1,...,x_{i},z_1,...,z_{\la(i-1)}\rg$ hence, by the Implicit Function Theorem as we did for $D_1^h$, $D_i^h$ is isomorphic to $\K \lg x_1,...,x_{i},z_1,...,z_{\la(i)}\rg$ for some new indeterminates $z_{\la(i-1)+1}$,..., $z_{\la(i)}$.
 
 Finally the morphisms $C_i\lgw D_i^h$ define the $y_k(x,z)$ satisfying $f(x,y(x,z))=0$. The power series $z_j(x)$ are defined by the morphisms $D_i^h\lgw B_i$ and the fact that $C_i\lgw B_i$ factors through $D_i^h$ yields $y(x)=y(x,z(x))$. 
\end{proof}

\medskip
\begin{proof}[Proof of Lemma \ref{lem}]
Let $\psi: A_i\otimes_{A_{i-1}}B\lgw B\lg x_i\rg$ be the morphism defined by 
$\psi(\sum_ja_j\otimes b_j)=\sum_j a_jb_j$ with $a_j\in A_i$ and $b_j\in B$ for any $j$. The morphism $\psi$ is well defined since $A_i$ and $B$ are subrings of the ring $B\lg x_i\rg$. The image of $\psi$ is the subring of $B\lg x_i\rg$ generated by $A_i$ and $B$.\\
Let us prove that $\psi$ is injective: Let $\sum_ja_j\otimes b_j\in \Ker(\psi)$ with $a_j\in A_i$ and $b_j\in B$ for any $j$. This means that $\sum_ja_jb_j=0$. Let us write $a_j=\sum_{l\in\N}a_{j,l}x_i^l$ where $a_{j,l}\in A_{i-1}$ for any $j$ and $l$. Thus we have 
\begin{equation}\label{eq1}\sum_ja_{j,l}b_j=0\end{equation} for any $l\in\N$ and this system of linear equations is equivalent to a finite system by noetherianity. The ring extension $A_{i-1} \lgw B$ is flat since $A_{i-1}\lgw \K [[x_1,...,x_{i-1}]]$ and  $B\lgw \K [[x_1,...,x_{i-1}]]$ are faithfully flat (they are completions of  local noetherian rings, cf. \cite{Ma} p. 46 and Theorem 8.14 p. 62). Thus the solution vector $(b_j)_j$ of \eqref{eq1}  is a linear combination with coefficients in $B$ of solution vectors in $A_{i-1}$ (cf. \cite{Ma} Theorem 7.6 p.49). Thus $(b_j)_j=\sum_k b'_k(a'_{j,k})_j$ where $b_k'\in B$  and, for any $k$, $(a'_{j,k})_j$ are vectors with entries in $A_{i-1}$ which are solutions of \eqref{eq1}.  This means that $$\sum_ja_j\otimes b_j=\sum_{j,k}a_j\otimes b_k'a'_{j,k}=\sum_k\sum_ja_ja'_{j,k}\otimes b'_k=\sum_k(\sum_l(\sum_ja_{j,l}a'_{j,k})x_i^l)\otimes b'_k=0.$$
Thus $\Ker(\psi)=(0)$.

Obviously $\Im(\psi)$ contains $B[x_i]$ whose henselization is $B\lg x_i\rg$ by Lemma \ref{lem2}, thus $\psi$ induces a surjective  morphism between the henselization of $A_{i}\otimes_{A_{i-1}}B$ and  $B\lg x_i\rg$. This surjective morphism is also injective since $\psi$ is injective and $A_{i}\otimes_{A_{i-1}}B$ is a domain (Indeed if $y\neq 0$ is in the henselization of $A_{i}\otimes_{A_{i-1}}B$, then $y$ is a root of a non zero polynomial with coefficients in $A_{i}\otimes_{A_{i-1}}B$. Since $A_{i}\otimes_{A_{i-1}}B$ is a domain and $y\neq 0$ we may assume that this polynomial has a non zero constant term denoted by $a$. If the image of $y$ in $B\lg x_i\rg$ is zero then $\psi(a)=0$ which is a contradiction).\\
On the other hand $B\lg x_i\rg$ is the henselization of $B[x_i]$ which is noetherian, thus $B\lg x_i\rg$ is noetherian (cf. \cite{EGA} Th\'eor\`eme 18.6.6.). This proves that the henselization of $A_{i}\otimes_{A_{i-1}}B$ is noetherian. Hence $A_{i}\otimes_{A_{i-1}}B$ is noetherian (cf. \cite{EGA} Th\'eor\`eme 18.6.6.).
\end{proof}
\medskip

\begin{proof}[Proof of Lemma \ref{lem2}]
The elements of the henselization of  a local ring $A$ are algebraic over $A$ by construction. Thus  the henselization of $\frac{B[x_i]}{I}$ is a subring of $\frac{B\lg x_i\rg}{I}$. \\
On the other hand let us prove first that $B\lg x_i\rg$ is the henselization of $B[x_i]_{(x_1,...,x_i)}$. If $y\in   B\lg x_i\rg$, then $y$ is a root  of a polynomial $P(Y)$ with coefficients in $B[x_i]$. By Artin approximation Theorem (see Theorem 11.3 \cite{Sp}), $y$ may be approximated by elements which are in the henselization of $B[x_i]$. Since $P(Y)$ has only a finite number of roots, this means that $y$ is in the henselization of $B[x_i]_{(x_1,...,x_i)}$. Thus $B\lg x_i\rg$ is the henselization of $B[x_i]_{(x_1,...,x_i)}$. Now the morphism $B[x_i]\lgw \frac{B[x_i]}{I}$ induces a morphism $B\lg x_i\rg\lgw \left(\frac{B[x_i]}{I}_{(x_1,...,x_i)}\right)^h$ of $B[x_i]$-algebras by the universal property of the henselization. It is clear that the kernel of this morphism is generated by $I$ thus we get an injective morphism $\frac{B\lg x_i\rg}{I}\lgw \left(\frac{B[x_i]}{I}_{(x_1,...,x_i)}\right)^h$ of $B[x_i]$-algebras. Since $\left(\frac{B[x_i]}{I}_{(x_1,...,x_i)}\right)^h$ is a subring of  $\frac{B\lg x_i\rg}{I}$, this shows that the morphism $\frac{B\lg x_i\rg}{I}\lgw \left(\frac{B[x_i]}{I}_{(x_1,...,x_i)}\right)^h$ is an isomorphism.
\end{proof}


\section{Algebraic Equisingularity of Zariski}\label{varchenko}
Notation:  Let $x=(x_1, \ldots, x_n) \in \C^n$. Then we denote $x^i = (x_1, \ldots, x_i) \in \C^i$.

\subsection{Assumptions} \label{pseudopolynomials}
Let $V$ be an analytic hypersurface of a neighborhood of the origin in
$\C^l \times  \C^n$ and let $W= V\cap (\C ^l \times  \{0\} )$.  
 Suppose there are given complex pseudopolynomials 
 \begin{align*}
F_{i} (t, x^i )= x_i^{p_i}+ \sum_{j=1}^{p_i} a_{i-1,j} (t,x^{i-1})
 x_i^{p_i-j}, \qquad    i=0, ... ,n,
\end{align*}
$t\in \C ^ l$, $x^i \in \C ^i$, 
with complex analytic coefficients $a_{i-1,j}$, that satisfy 
\begin{enumerate}
\item 
$V=F_n^ {-1} (0)$.
\item
$F_{i-1} (t, x^{i-1})=0$  if and only if $F_{i} (t, x^{i-1}, x_i)=0$ considered as an equation on $x_i$ with 
$(t, x^{i-1})$ fixed, has fewer roots than for generic $(t, x^{i-1})$. 
\item
$F_0\equiv 1$. 
\item
There are positive reals $\delta_k>0$, $k=1, \ldots, l$, and $\varepsilon _j> 0$, $j=1, \ldots, n$,  such that $F_i$ are defined on the polydiscs $U_i:= \{ |t_k|< \delta_k, |x_j|< \varepsilon _j , k=1, \ldots,l,   j=1, \ldots, i \}$.
\item
All roots of $F_{i} (t, x^{i-1}, x_i)=0$, for $(t,x^{i-1}) \in U_{i-1}$,  lie inside the circle of radius $\varepsilon _i$.  
\item
Either $F_i(t,0)\equiv 0$ or $F_i \equiv 1$ (and in the latter case 
$F_k \equiv 1$ for all $k\le i$).
\end{enumerate}

We may take as $F_{i-1}$ the Weierstrass polynomial associated to the reduced discriminant of $F_{i}$   
or a generalized discriminant (see the next section).  

We shall denote $V_i=F_i ^{-1} (0) \subset U_i$.  
For the parameter $t$ fixed we write $V_t := V\cap (\{t\}\times \C^n)$,  $V_{i,t} := V_i\cap (\{t\}\times \C^i)$, and $U_{i,t} = U_i 
\cap (\{t\}\times \C^i)$.  We identify $W$ and $U_0$.

\begin{thm}\label{theorem} {\rm (\cite{varchenko1973} Theorem 1, \cite{varchenkoICM} Theorem 1)}
Under the above assumptions 
$V$ is topologically equisingular along $W$ with respect to the family of sections $V_t = V\cap (\{t\}\times \C^n)$. 
 This means that for all $t\in W$ there is a homeomorphism 
 $h_t : U_{n,0} \to U_{n,t}$ such that  $h_t (V_0)= V_t$ and $h_t(0)=0$. 
\end{thm}

\subsection{Remarks on Varchenko's proof of Theorem \ref{theorem}. } \label{rmk_Varchenko}

As Varchenko states in Remark 1 of \cite{varchenko1973} a stronger result holds, 
the family $V_t$ is topologically trivial, in the sense that the homeomorphisms $h_t$ depend continuously on 
$t$.  The details of the proof of 
Theorem \ref{theorem}  (with continuous dependence of $h_t$ on $t$) are published in  \cite{varchenko1972}.   

The homeomorphisms $h_t$ are constructed in  \cite{varchenko1972}  inductively by lifting step by step the homeomorphisms 
$$
 h_{i,t} : U_{i,0}  \to U_{i,t} ,  
 $$
 so that $ h_{i,t} (x^{i-1}, x_i) = ( h_{i-1,t} (x^{i-1}),  h_{i,t,i}  (x^{i})) $, $ h_{i,t}(V_{i,0})= V_{i,t}$, $ h_{i,t}(0)= 0$.     If $h_{i-1,t}$ depends continuously on $t$, then the number of  roots  of 
 $F_{i} (h_{i-1,t} (x^{i-1}), x_i)=0$  is independent of $t$.   Therefore, if $F_n= G_1 \cdots G_k$, then the 
 number of roots of each  $G_{j} (h_{n-1,t} (x^{n-1}), x_n)=0$ is independent of $t$, see Lemma 2.2 of  \cite{varchenko1972}.  In particular $h_t $ preserves not only $V=F_n ^{-1}(0)$ but also each of $G_j ^{-1}(0)$.   
Thus  \cite{varchenko1972}  implies the following. 
 
 \begin{thm}\label{theorem2} 
The homeomorphisms $h_t$ of Theorem \ref{theorem} can be chosen continuous in $t$.  If $F_n= G_1 \cdots G_k$ then for each $s=1, \ldots , k$, 
$h_t (G_s^ {-1} (0) \cap  (\{0\}\times \C^n) )=G_s^ {-1} (0) \cap  (\{t\}\times \C^n)$. 
\end{thm}

\section{Mostowski's Theorem.}\label{mostowskisection}

In this section we show Theorem \ref{homeotoalgebraic}

\subsection{Generalized discriminants}
Let $f(T) = T^p + \sum_{j=1}^ p a_iT^{p-i} = \prod_{j=1}^ p (T-T_i)$.  Then the expressions 
$$
 \sum_{r_1, \ldots r_{j-1}} \,  \prod_{k< l, k,l \ne r_1, \ldots r_{j-1}}  (T_k-T_l)^2
$$
are symmetric in $T_1, \ldots , T_p$ and hence polynomials in $a=(a_1, \ldots, a_{p}) $.  We denote 
these polynomials by $\Delta_j (a)$.  
Thus $\Delta_1$ is the standard discriminant and $f$ has exactly $p-j$ distinct roots if and only if 
$\Delta_1=\cdots = \Delta _{j}=0$ and $\Delta_{j+1}\neq 0$.


\subsection{Construction of a normal system of equations}
 
 Let be given a finite set of  pseudopolynomials $g_1, \ldots, g_k \in\C\{x\}$:  
   \begin{align*}
g_{s} ( x)= x_n^{r_s}+ \sum_{j=1}^{r_s} a_{n-1,s,j} (x^{n-1}) x_n^{r_s-j}.
\end{align*}
The coefficients $a_{n-1,s,j}$ can be arranged in a row vector  
$a_{n-1} \in \C\{x^{n-1}\}^{p_n}$ where $p_n:=\sum_s r_s$.  
Let $f_n$ be the product of the $g_s$'s.  The generalized discriminants $\Delta_{n,i} $ of $f_n$ are 
polynomials in $a_{n-1}$.  
  Let $j_n$ be a positive integer such that   \begin{align}\label{discriminants:n}
\Delta_{n,i} ( a_{n-1} )\equiv 0 \qquad i<j_n   ,
\end{align}
and  $\Delta_{n,j_n}  ( a_{n-1} ) \not \equiv 0$.  
 Then, after a linear change of coordinates $x^{n-1}$, we may write 
    \begin{align*}
 \Delta_{n,j_n} ( a_{n-1} ) =  u_{n-1} (x^{n-1}) (x_{n-1}^{p_{n-1}}
 + \sum_{j=1}^{p_{n-1}} a_{n-2,j} (x^{n-2}) x_{n-1}^{p_{n-1}-j} ) 
 .
\end{align*}
where $u_{n-1}(0)\ne 0$ and for all $j$, $a_{n-2,j}(0)=0$.  We denote $$
f_{n-1} = x_{n-1}^{p_{n-1}}+
  \sum_{j=1}^{p_{n-1}} a_{n-2,j} ( x^{n-2}) x_{n-1}^{p_{n-1}-j} $$
    and the vector 
of its coefficients $a_{n-2,j}$ by $a_{n-2} \in \C\{x^{n-2}\}^{ p_{n-1}}$.   
Let $j_{n-1}$ be the positive integer such that the first $j_{n-1}-1$ generalized discriminants $\Delta_{n-1,i} $ 
of $f_{n-1}$ are identically zero and  $\Delta_{n-1,j_{n-1}} $ is not.  Then again we define 
$f_{n-2} ( x^{n-2})$ as the Weierstrass polynomial associated to 
 $\Delta_{n-1,j_{n-1}}  $. 
 
We continue this construction and  
 define a sequence of pseudopolynomials $f_{i} (  x^i )$, $i=1, \ldots, n-1$, such that 
 $f_i= x_i^{p_i}+ \sum_{j=1}^{p_i} a_{i-1,j} (x^{i-1}) x_i^{p_i-j}  $ is the Weierstrass polynomial associated to the first non identically zero generalized discriminant $\Delta_{i+1,j_{i+1}} ( a_{i} )$ of $f_{i+1}$, 
where  we denote in general $a_{i}= (a_{i,1} , \ldots , a_{i,p_{i+1}} )$, 
  \begin{align}\label{polynomials:f_i}
 \Delta_{i+1,j_{i+1}} ( a_{i} ) =  u_{i} (x^{i})  (x_i^{p_i}+ \sum_{j=1}^{p_i} a_{i-1,j} (x^{i-1}) x_i^{p_i-j}  ) ,  
 \quad i=0,...,n-1.  
\end{align}
 Thus the vector  
of functions $a_i$ satisfies  
  \begin{align}\label{discriminants:i}
\Delta_{i+1,k} ( a_{i} )\equiv 0 \qquad k<j_{i+1}   ,  \quad i=0,...,n-1. 
\end{align}

This means in particular that 
  \begin{align*}
\Delta_{1,k} ( a_{0} ) \equiv 0 \quad \text {for } k<j_1  \text { and }  \Delta_{1,j_1} ( a_{0} ) \equiv u_0 ,
\end{align*}
where $u_0$ is a non-zero constant.


\subsection{Approximation by Nash functions} Consider \eqref{polynomials:f_i} and 
\eqref{discriminants:i} as a system of polynomial equations on $a_i(x^i)$, $u_i (x^i)$.  By construction,  this system admits convergent solutions.  Therefore, by Theorem \ref{nest_ploski},   there exist a new set of variables $z=(z_1,..., z_s)$, an increasing function $\tau$, and  convergent power series $z_i(x)\in\C \{x\}$ 
{vanishing at 0} such that 
$z_1(x)$,..., $z_{\tau(i)}(x)$ depend only on $(x_1,..., x_{i})$, algebraic power series 
$u_i (x^i,z)  \in\C \langle x^{i},z_1,...,z_{\tau(i)}\rangle$ and  vectors of algebraic power series 
$a_i(x^i,z) \in\C \langle x^{(i)},z_1,...,z_{\tau(i)}\rangle^{p_i}$, such that $a_i(x^i,z)$, $u_i (x^i,z)$ are solutions of  \eqref{polynomials:f_i},  
\eqref{discriminants:i} and $a_i(x^i)= a_i(x^i,z(x^i))$,  
$u_i(x^i)= u_i(x^i,z(x^i))$.  

For $t\in \C$ we define 
\begin{align*}
 F_n(t,x) = \prod_s G_s(t,x)  , \quad G_{s} (t, x)=  x_n^{r_s}+ 
 \sum_{j=1}^{r_s} a_{n-1,s,j} (x^{n-1}, t z(x^{n-1})) x_n^{r_s-j}  
\end{align*}
 \begin{align*}
 F_i(t,x) = x_i^{p_i}+ \sum_{j=1}^{p_i} a_{i-1,j} (x^{i-1}, t z(x^{i-1} ) x_i^{p_i-j} , \quad i=0, ... , n-1.
\end{align*}
Finally we set $F_0\equiv 1$.  
Because $u_i(0,0)= u_i(0,z(0))\ne 0$,  the family $F_{i} (t, x)$ satisfies the assumptions of Theorem \ref{theorem} with $|t|<R$ for any $R<\infty$.

\begin{cor} \label{homeotoNash}  
Let  $(V,0) \subset (\K^n,0)$ be an analytic germ defined by $g_1=...  =g_k=0$ 
with $g_s \in \K \{x\}$.  
Then there are algebraic power series $\hat g_s \in \K \langle x\rangle$  and 
a homeomorphism germ $h: (\K^n,0) \to (\K^n,0)$ such that $h (g_s ^{-1}(0)) = \hat  g_s ^{-1} (0) $ for 
$s=1, ...  ,k$.  In particular,   $h(V)$ is the Nash set germ $ \{\hat g_1= ...  = \hat g_k=0\}$.  
\end{cor}

\begin{proof}
For $\K =\C$ we set $\hat g_i(x) = G_i (0,x)$ and then the corollary follows from Theorem \ref{theorem2}.  \\ 
The real case follows from the complex one because if the pseudopolynomials $F_i$ of subsection \ref{pseudopolynomials} have real coefficients then the homeomorphisms $h_t$ constructed
 in \cite{varchenko1972} are conjugation invariant, cf.  \S 6 of.  \cite{varchenko1972}. 
\end{proof}

Now Theorem \ref{homeotoalgebraic} follows from Corollary \ref{homeotoNash} and the following result.

\begin{thm} \label{Nashtoalgebraic}  {\rm (\cite{bochnakkucharz1984} Theorem 2.)}
 Let  $(V,0) \subset (\K^n,0)$ be a Nash set germ.  Then there is a local Nash 
diffeomorphism $\sigma: (\K^n,0) \to (\K^n,0)$ such that 
$\sigma(V)$ is the germ of an algebraic subset of $\K^n$.  
\end{thm}



\section{Topological equivalence between analytic and algebraic function germs}\label{functiongerms}

In this section we show Theorem  \ref{homeotopolynomial} and Theorem \ref{generalhomeo}.


\subsection{A variant of Varchenko's method} \label{pseudopolynomials2}

We replace the assumptions (2) and (3) of Subsection \ref{pseudopolynomials} by
 \begin{enumerate}
\item [(2')]
There are $q_i\in \N$ such that $x_1^{q_i} F_{i-1} (t, x^{i-1})=0$  if and only if the equation $F_{i} (t, x^{i-1}, x_i)=0$, has fewer roots than for generic $(t, x^{i-1})$. 
\item [(3')]
$F_1\equiv 1$. 
\end{enumerate}

Then Varchenko's method gives the following result. 

\begin{thm}\label{theoremy} 
Under the above assumptions,  $V$ is topologically equisingular along $W$ with respect 
to the family of sections $V_t = V\cap (\{t\}\times \C^n)$.   Moreover all the sections $V\cap \{x_1=const\}$ are also equisingular.  
 This means that for all $t\in W$ there is a homeomorphism 
 $h_t : U_{n,0} \to U_{n,t}$ such that  $h_t (V_0)= V_t$, $h_t(0)=0$, and $h_t$ preserves the levels of $x_1$ 
 \begin{align}\label{h_t}
 h_t(x_1, ...,x_n) = (x_1, \hat h_t(x_1, ...,x_n)).
 \end{align}
\end{thm}

Indeed, recall that the homeomorphisms $h_t$ are constructed inductively by lifting step by step the homeomorphisms $ h_{i,t} : U_{i,0}  \to U_{i,t} $, so that 
 $ h_{i,t} (x^{i-1}, x_i) = ( h_{i-1,t} (x^{i-1}),  h_{i,t,i}  (x^{i})) $.  
At each stage such  lifts  $h_{i,t}$ exist and preserve  
the zero set of $F_i$ if $h_{i-1,t}$ depends continuously on $t$ and preserves the discriminant set of 
$F_i$, see \cite{varchenko1972} sections 2 and 3.   

Because $F_1\equiv 1$, by (2'), the discriminant set of $F_2$ is either empty or given by $x_1=0$.  
Therefore we may take $h_{1,t}(x_1) = x_1$.    Then we show by induction on $i$  that each $h_{i,t}$ 
can be lifted so that the lift $h_{i+1,t}$ preserves the zero set of $F_{i+1}$ and the values of $x_1$.  The former 
condition follows by inductive assumption, $h_{i,t}$ preserves the discriminant set of $F_{i+1}$.  
The latter condition is satisfied trivially since $h_{i+1,t}$ is a lift of $h_{i,t}$.

\subsection{Equisingularity of functions.  }
We apply Theorem \ref{theoremy} to study the equisingularity of analytic function germs as follows.  
Let $G (t,y) : (\C^l\times \C^{n-1},0) \to (\C,0)$ be analytic, $y=(y_1, ...,y_{n-1})$.  
We associate to $G$ its graph $V= \{(t,x_1,x_2,...,x_n) ; x_1= G_t(x_2,...,x_n)\}$, thus fixing the following 
notation 
\begin{align}\label{notation}
x= (x_1,x_2,...,x_n)= (x_1, y)
\end{align}
 We consider $G$ as an 
analytic family of analytic function germs $G_t : (\C^{n-1},0) \to (\C,0)$ parametrized by 
$t\in W$, where $W$ is a neighborhood of the origin in $\C^l$.

\begin{thm}\label{theoremg} 
 Suppose that $V$  and $W$ satisfy 
 the assumptions of Theorem \ref{theoremy}.  
 Then the family of analytic function germs $G_t$ is topologically equisingular.  
 This means that there is a family of local homeomorphisms $\sigma_t : (\C^{n-1},0) \to \C^{n-1},0)$ such that 
 $$
 G_0 = G_t \circ \sigma_t .
 $$
\end{thm}

\begin{proof}
It follows from \eqref{h_t} by setting 
$\sigma_t(y) = \hat h_t (G_0(y),y)$.  Indeed, since $h_t$ preserves $V$ we have 
$$
h_t (G_0(y),y) = (G_t (\hat h_t (G_0(y),y)), \hat h_t (G_0(y),y)),
$$
and since it preserves the levels of $x_1$ 
$$
G_0(y) = G_t (\hat h_t (G_0(y),y))  .
$$
\end{proof}


\subsection{Construction of a normal system of equations for a finite family of function germs.}\label{construction:new}
 
Let $g_m :(\C^{n-1},0)\rightarrow(\C,0)$, $m=1, ...,p$, be a finite family of analytic function germs that we assume not identically equal to zero.   
After a linear change of coordinates $(x_2, ...,x_n)$, $\prod_{m=1}^p(x_1-g_m (x_2,...,x_n))$ is equivalent to a 
pseudopolynomial, that is we may write 
 $$
\prod_{m=1}^p(x_1-g_m(x_2,...,x_n))= u_{n}  (x)(x_n^{p_n}+ \sum_{j=1}^{p_n} a_{n-1,j} (x^{n-1}) x_n^{p_n-j}), 
$$
where $u_{n}(0)\ne 0$ and $a_{n-1,j}(0)=0$.  We denote 
$$f_{n} (x)= x_n^{p_n}+ \sum_{j=1}^{p_n} a_{n-1,j} (x^{n-1}) x_n^{p_n-j}$$ so that 
\begin{align}\label{new2}
u_n(x) f_n(x)=\prod_{m=1}^p(x_1-\sum_{k=2}^nx_kb_{m,k}(x_2,...,x_n))
\end{align}
with $g_m=\sum_{k=2}^nx_kb_{m,k}$.
We denote by $b  \in\C \{ x\} ^{p(n-1)}$ the vector of the coefficients $b_{m,k}$
 and by $a_{n-1} \in \C\{x^{n-1}\}^{ p_{n}}$ the one
 of the coefficients  $a_{n-1,j}$.

 The generalized discriminants $\Delta_{n,i} $ of $f_n$ are 
polynomials in $a_{n-1}$.  
  Let $j_n$ be a positive integer such that  
 $$
\Delta_{n,i} ( a_{n-1} )\equiv 0 \qquad i<j_n   ,
$$
and  $\Delta_{n,j_n}  ( a_{n-1} ) \not \equiv 0$.  
 After a change of coordinates $(x_2, ...,x_{n-1})$ we may write 
$$
 \Delta_{n,j_n} ( a_{n-1} ) =  u_{n-1}  (x^{n-1})  x_1^{q_{n-1}} (x_{n-1}^{p_{n-1}}+
  \sum_{j=1}^{p_{n-1}} a_{n-2,j} (x^{n-2}) x_{n-1}^{p_{n-1}-j} ) , 
$$
where $u_{n-1}(0)\ne 0$ and $a_{n-2,j}(0)=0$.  We denote 
$
f_{n-1} = x_{n-1}^{p_{n-1}}+
  \sum_{j=1}^{p_{n-1}} a_{n-2,j} (x^{n-2}) x_{n-1}^{p_{n-1}-j} $
    and the vector 
of its coefficients $a_{n-2,j}$ by $a_{n-2} \in \C\{x^{n-2}\}^{ p_{n-1}}$.  
Let $j_{n-1}$ be the positive integer such that the first $j_{n-1}-1$ generalized discriminants $\Delta_{n-1,i} $ 
of $f_{n-1}$ are identically zero and  $\Delta_{n-1,j_{n-1}} $ is not.  Then again we divide  $\Delta_{n-1,j_{n-1}}  $ 
by the maximal power of $x_1$ and, after a change of coordinates $(x_2, ...,x_{n-2})$, denote the associated 
Weierstrass polynomial by $f_{n-2} ( x^{n-2})$.  

We continue this construction and  
 define a sequence of pseudopolynomials $f_{i} (  x^i )$, $i=1, \ldots, n-1$, such that 
 $f_i= x_i^{p_i}+ \sum_{j=1}^{p_i} a_{i-1,j} (x^{i-1}) x_i^{p_i-j}  $ is the Weierstrass polynomial associated to the first non identically zero generalized discriminant $\Delta_{i,j_i} ( a_{i+1} )$ of $f_{i+1}$,  divided by the maximal power of $x_1$, 
where  we denote in general $a_{i}= (a_{i,1} , \ldots , a_{i,p_{i}} )$, 
  \begin{align}\label{polynomialsnew:f_i}
 \Delta_{i+1,j_{i+1}} ( a_{i} ) =  u_{i} (x^{i})  x_1^{q_i} (x_i^{p_i}+ \sum_{j=1}^{p_i} a_{i-1,j} (x^{i-1}) x_i^{p_i-j}  ) ,  
 \quad i=0,...,n-1.  
\end{align}
 Thus the vector  
of functions $a_i$ satisfies  
  \begin{align}\label{discriminantsnew:i}
\Delta_{i+1,k} ( a_{i-1} )\equiv 0 \qquad k<j_{i+1}   ,  \quad i=0,...,n-1. 
\end{align}

These equations mean in particular that 
  \begin{align}\label{discriminantsnew:0}
\Delta_{1,k} ( a_{0} ) \equiv 0 \quad \text {for } k<j_1  \text { and }  \Delta_{1,j_1} ( a_{0} ) \equiv u_0 x_1^{q_0}  .
\end{align}
where $u_0$ is a non-zero constant.  Hence $f_1 \equiv 1$.  


\subsection{Approximation by Nash functions} \label{approximation2}
Consider  \eqref{new2}, \eqref{polynomialsnew:f_i}, \eqref{discriminantsnew:i}, as a system of polynomial equations on 
$a_i(x^i)$, $u_i(x^i)$, and $b(x)$.   By construction,  this system admits convergent solutions.  Therefore, by Theorem \ref{nest_ploski},   there exist a new set of variables $z=(z_1,..., z_s)$, an increasing function $\tau$, and  convergent power series $z_i(x)\in\C \{x\}$  vanishing at 0  such that 
$z_1(x)$,..., $z_{\tau(i)}(x)$ depend only on $(x_1,..., x_{i})$, algebraic power series 
$u_i (x^i,z)  \in\C \langle x^{i},z_1,...,z_{\tau(i)}\rangle$,  
and  vectors of algebraic power series 
$a_i(x^i,z) \in\C \langle x^{(i)},z_1,...,z_{\tau(i)}\rangle^{p_i}$, $ b(x,z)  \in\C \langle x,z\rangle ^{n-1}$, 
such that 
$a_i(x^i,z)$, $u_i (x^i,z)$, $ b(x,z)$, are solutions of  \eqref{new2}, \eqref{polynomialsnew:f_i},  
\eqref{discriminantsnew:i} and  $a_i(x^i)= a_i(x^i,z(x^i))$,  
$u_i(x^i)= u_i(x^i,z(x^i))$, $b(x)=  b(x,z(x))$.

For $t\in \C$ we define 
$$
 F_i(t, x) = x_i^{p_i}+ \sum_{j=1}^{p_i} a_{i-1,j} ( x^{i-1}, tz(x^{i-1})) x_i^{p_i-j} . 
$$
In particular, by \eqref{discriminantsnew:0},   $F_1\equiv 1$.  
Since  
$$u_n(x,tz(x))  F_n(t, x) = \prod_{m=1}^p (x_1-\sum_{k=2}^nx_kb_{m,k}(x,tz(x))),$$  
 by the Implicit Function Theorem 
there are algebraic power series  
$G_m \in  \C \langle t, y_1. ... , y_{n-1}\rangle$ such that 
$$
F_n^{-1}(0)= \bigcup_m \{(t,x); x_1 = G_m (t,x_2,...,x_n)\}
$$
as germs at the origin.  Then $g_m(y) = G_m (1,y)$ and $G_m (0,y)\in   \C \langle  y \rangle$.  
We denote  $\hat g_m(y) = G_m (0,y)$.  

Because $u_i(0,0)= u_i(0,z(0))\ne 0$,  the family $F_{i} (t, x)$ satisfies the assumptions of Theorem \ref{theoremy} with $|t|<R$ for arbitrary $R<\infty$.  
 By Theorem \ref{theoremy} there is 
a continuous family of homomorphism germs $h_t: (\C^n,0) \to (\C^n,0)$, $h_t (x) = (x_1, \hat h_t (x_1,x_2,...,x_n))$, 
such that 
$$
h_t (g_m (y),y) = (G_m (t,\hat h_t (g_m (y),y)), \hat h_t (g_m(y),y)),
$$
Fix one $m$, for instance $m=1$, and set 
$$
\sigma_{t} (y) =  \hat h_t (g_1(y),y)
$$
as in the proof of Theorem \ref{theoremg} (we use here the notation \eqref{notation}).   
Then $ g_1(y) = G_1 (t, \sigma_t  (y))$  and in particular 
\begin{align}\label{homeofunction}
 g_1(y) = \hat g_1 (\sigma_0 (y)).
\end{align}
It is not true in general that  $g_m(y) = \hat g_m (\sigma_0 (y))$  since the homeomorphism $\sigma_t$ is 
defined by restricting $h_t$ to the graph of $G_1$.  If we define 
$$
\sigma_{m,t} (y) =  \hat h_t (g_m(y),y)
$$
then we have 
$$g_m(y) = \hat g_m (\sigma_{m,0} (y))$$
Both homeomorphisms coincide on $X_{m} = \{y\in (\C^{n-1},0); ( g_m-g_1) (y)=0\}$.  Therefore if we define 
$\hat X_{m} = \{y\in (\C^{n-1},0); ( \hat g_m- \hat g_1) (y)=0\}$ then 
\begin{align}\label{homeoset}
\sigma_0(\hat X_{m}) = X_{m}.
\end{align}

    Therefore we have the following result.  

\begin{prop} \label{homeotoNash2}  
Let  $(V_i,0) \subset (\K^n,0)$ be a finite family of analytic set germs and let  $g: (\K^n,0)\to (\K, 0)$ be an analytic function germ.   
Then there are Nash set germs $(\hat V_i,0) \subset (\K^n,0)$,  
an algebraic power series $\hat g \in \K \langle x\rangle$,    and 
 a homeomorphism germ $\hat \sigma : (\K^{n-1},0) \to (\K^{n-1},0)$ such that $\sigma ( \hat  V_i) = 
V_i$ and $g\circ \hat \sigma= \hat g$.   
\end{prop}

\begin{proof}
Let $\K =\C$.  Choose a finite family $g_m :(\C^{n-1},0)\rightarrow(\C,0)$, $m=1, ...,p$, of analytic function 
such that $g_1=g$ and for every $i$, the ideal of $V_i$  is generated by some of the differences 
$g_m-g_1$.  We apply to the family $g_m$ the procedure of subsections \ref{construction:new} and \ref{approximation2} and set $\hat \sigma = \sigma_0$.  The claim now follows from \eqref{homeofunction} and 
\eqref{homeoset}.    

The real case follows from the complex one because if the pseudopolynomials $F_i$ of subsection \ref{pseudopolynomials} have real coefficients then the homeomorphisms $h_t$ constructed
 in \cite{varchenko1972} are conjugation invariant, cf.  \S 6 of.  \cite{varchenko1972}. 
\end{proof}


\subsection{Proof of Theorem  \ref{homeotopolynomial} and Theorem \ref{generalhomeo} }

It suffices to show Theorem \ref{generalhomeo}.  It will follow from Proposition \ref {homeotoNash2} and the 
next two results.

\begin{thm}\label{bochnakkucharzbetter}
Let $\K = \R$ or $\C$.  Let  $f_i : (\K^{n},0)\to (\K, 0)$,  be a finite family of Nash function germs.  
Then there is a Nash diffeomorphism $ h : (\K^n,0) \to (\K^n,0)$ and analytic (even Nash) units 
$u_i : (\K^{n},0)\to \K$, $u_i(0)\ne 0$, such that for all $i$,  $u_i(x) f_i (h(x))$  are 
germs of polynomials.    
\end{thm}   

\begin{proof}
For $\K=\C$ Theorem \ref{bochnakkucharzbetter} follows from  Theorem 5 of \cite{bochnakkucharz1984}.  Indeed, (i) implies (ii) of this theorem  gives :\\
\emph{If $(V,0) \subset  (\K^{n},0)$ is a Nash set germ then there is a Nash diffeomorphism 
$ h : (\K^n,0) \to (\K^n,0)$ such that for any analytic irreducible component $W$ of $(V,0)$ the ideal of functions 
vanishing on $h(W)$ is generated by polynomials.}

Now, if $\K=\C$, it suffices to apply the above result to $(V,0)$ defined as the zero set of the product of $f_i$'s.  

If $\K=\R$ such a set theoretic statement is not sufficent but in this case 
Theorem \ref{bochnakkucharzbetter} follows from the proof of Theorem 5 of \cite{bochnakkucharz1984}. 
 We sketch this argument below.  
 
 First we consider $\K=\C$.  
 Choose representatives $f_i : U \to \C$ of the germs $f_i$, $i=1, ..., m$, and let $f=(f_1, ... , f_m): U \to \C^m$.  
 By Artin-Mazur Theorem, \cite{BCR} Theorem 8.4.4,   \cite{bochnakkucharz1984} Proposition 2, there is an algebraic set $X\subset \C^n \times \C^N$ of dimension $n$, a polynomial map 
 $\Phi : \C^n \times \C^N \to \C^m$, and a Nash map $s: U \to X$ such that $f=\Phi \circ s$,  $s: U \to s(U)$ is 
 a Nash diffeomorphism and $s(U) \cap Sing (X) = \emptyset$.  ($X$ is the normalization of the Zariski closure of 
 the graph of $f$.)    
We may assume that $p=s(0)$ is the origin in $ \C^n \times \C^N $.  

Let   $\pi : X\to \C^n$ be a  generic linear projection.   
Then the germ $h$ of $ (\pi \circ s)^{-1} $ satisfies the claim.   Indeed, denote $X_i = X\cap \Phi_i^{-1} (0)$.   Then for each $i=1, ..., m$, $Z_i=\pi(X_i) $ is an algebraic subset of $\C^n$ and moreover, $\pi$ induces 
a local isomorphism 
$(X_i, 0)\to (Z_i,0)$.  We fix a reduced polynomial $P_i$ that defines $Z_i$.  Then 
$f_i \circ h $, as a germ at the origin, vanishes exactly on $Z_i$ and hence equals a power of $P_i$ times an 
analytic unit.  

If $\K=\R$ then we apply the complex case to the complexifications of  the $f_i$'s keeping the construction 
conjugation invariant.  In particular the linear projection can be chosen real (that is conjugation invariant).  
Indeed, this projection is 
from the Zariski open dense subset  $\mathcal U$ of the set of linear projections $L (\C^n \times \C^N , \C^n)$.  
Then $\mathcal U \cap L (\R^n \times \R^N , \R^n) $ is non-empty. (A complex polynomial of $M$ variables  that vanishes on $\R^M$ is identically equal to zero).   This ends the proof.  
\end{proof}

\begin{thm}\label{homeoifunit}
Let $\K = \R$ or $\C$.  Let  $f: (\K^{n},0)\to (\K, 0)$ be an analytic function germ and let 
$u: (\K^{n},0)\to \K$ be an analytic unit, $u(0)\ne 0$ ($u(0)>0$ if $\K=\R$).  Let $(V_i,0) \subset 
(\K^{n},0)$ be a finite family of analytic set germs.  
Then there is a homeomorphism germ $ \sigma : (\K^n,0) \to (\K^n,0)$ such that $(\sigma (V_i),0) =(V_i,0)$ 
for each $i$ and $uf= f \circ  \sigma$.    
\end{thm}

\begin{proof}
If $\K=\C$ we suppose additionally that the segment that joins $u(0)$ and $1$ does not contain $0$.  The general case can be reduced to this one.  

Fix a small neighborhood $U$ of the origin in $\K^n$ so that the representatives $V_i\subset U$, $f:U\to \K$, 
and $u: U\to \K$ are well-defined.  In the proof we often shrink $U$ when necessary.  Let $I$ denote a small neighborhood of $[0,1]$ in $\R$.  We construct a Thom stratification of the deformation 
$\Psi(x,t)= (F(x,t),t): U\times I \to \K \times I$, where 
$$
F(x,t) = f(x)(1-t + t u(x))  = f(x)(1+ t (u(x)-1))  ,  \quad (x,t)\in U\times I, 
$$
that connects $f(x) = F(x,0)$ and $u(x)f(x) = F(x,1)$.  
 Then we conclude by the second Thom-Mather Isotopy Lemma.  For the Thom 
stratification we refer the reader to \cite{gibson}, Ch. 1, and for the Thom-Mather Isotopy Lemmas to 
\cite{gibson}, Ch. 2.

Fix a Thom stratification $\mathcal S' = \{S'_j\}$ of $f:U\to \K$  such that each $V_i$ is a union of strata.  
That means that $\mathcal S'$ is a Whitney stratification of $U$, compatible with $f^{-1} (0)$  and $f^{-1} (\K \setminus \{0\})$, that satisfies Thom's $a_f$ condition.  (It is well-known that such a stratification exists, 
the existence of $a_f$ regular stratifications was first proved in the complex analytic case by H. Hironaka in \cite{hironaka},  using resolution of singularities, under the assumption "sans \'eclatement" which is always satisfied for functions.  In the real subanalytic case it was first shown in \cite{kurdykaraby}) 

First we show that $\mathcal S= \{S_j=S'_j\times I\}$ as a stratification of $U\times I$ satisfies $a_F$ condition   
\begin{itemize}
\item [($a_F$)] 
\emph{for every stratum $S \subset F^{-1} (\K \setminus \{0\})$ and every sequence of points 
$p_i = (x_i,t_i) \in S$ that converges to a point $p_0= (x_0,t_0)\in S_0 \subset F^{-1} (0)$,  such that 
$\ker d_{p_i} F|_{T_ {p_i}S}\to T$, we have $T\supset  T_ {p_0}S_0$.}
\end{itemize}
By the curve selection lemma it suffices to check this condition on every real analytic curve 
$p(s)= (x(s),t(s)) : [0, \varepsilon) \to S\cup S_0$, $p(0)\in S_0$ and 
$p(s) \in S$ for $s> 0$.  Since $\mathcal S'$ satisfies $a_f$, 
the condition $a_F$ for $\mathcal S$ follows from the following lemma.  

\begin{lem}
Let 
 $S=S'\times I  \subset F^{-1} (\K \setminus \{0\})$ and  $S_0 \subset F^{-1} (0)$ be two strata of $\mathcal S$ and 
let $p(s)= (x(s),t(s)) : [0, \varepsilon) \to U\times I$ be a real analytic curve such that $p_0=p(0)\in S_0$ and 
$p(s) \in S$ for $s> 0$.  Then for $s>0$ and small, $\grad_x F|_S(p(s)) \ne 0$ and 
\begin{align}\label{limits}
\lim_{s\to 0}  \frac { \grad F|_S(p(s))}  {\| \grad F|_S(p(s))\|} = \lim_{s\to 0}  \frac { (\grad f|_{S'}(x(s)), 0) }
{ \| \grad f|_{S'}(x(s)) \| }
\end{align}
\end{lem}

\begin{proof}
By assumption $f(x(s)) = \sum_{i=i_0}^ \infty a_i s^i$, with $i_0>0$ and $a_{i_0}\ne 0$.    By differenting we obtain 
$$
|\frac {df}{ds}| =| \prodscal{\grad f|_{S'}} {x'(s)} | \le \|\grad f|_{S'}\| \| x'(s)\| .
$$
Hence there exists $C>0$ such that for small $s>0$ 
\begin{align}\label{bound}
|f(x(s))| \le s C \|\grad f|_{S'}\| .
\end{align}
Moreover
\begin{align}\label{gradients}
\grad F|_S(p(s)) = & (\grad f|_{S'}(x(s)), 0) (1+t(s)(u(x(s))-1) \\ &+ f(x(s)) (t(s) \grad u|_{S'} (x(s)), u(x) -1) .
\end{align}

Now \eqref{limits} follows easily from \eqref{bound} and \eqref{gradients}. 
\end{proof}

Finally $\mathcal S$ as a stratification of $U\times I$ together with  $((\K\setminus \{0\})\times I, \{0\}\times I)$ 
as a stratification of 
$\K\times I$ is a Thom stratification of $\Psi$.  Indeed,  $\mathcal S$ is a Whitney 
stratification as the product of a Whitney stratification of $U$ times $I$.  Secondly, for any pair of strata 
$S=S'\times I \subset F^{-1} (\K \setminus \{0\})$ and $S_0=S_0'\times I  \subset F^{-1} (0)$ it satisfies $a_F$ and hence also  $a_\Psi$ condition.  
 Therefore Theorem \ref{homeoifunit} follows from the second Thom-Mather Isotopy Lemma,  \cite{gibson}, Ch. 2 (5.8).  
\end{proof}

Now we may conclude the proof of Theorem \ref{generalhomeo}.  
By Proposition \ref {homeotoNash2} we may assume that $g$ is a Nash function germ and the $V_i$'s are Nash sets germs.  Moreover by Theorem \ref{bochnakkucharzbetter}, after composing with the Nash diffeomorphism 
$h$, we may assume that $g$ equals a polynomial times an analytic unit and that the ideal of analytic function 
germs defining each $V_i$ is generated by polynomials.  In particular each $V_i$ is algebraic.  Finally we 
apply Theorem \ref{homeoifunit} to show that, after composing with a homeomorphism preserving 
each $V_i$, $g$ becomes a polynomial.   



\section{Examples}\label{examples}

\begin{example}\label{whitney1}
We give here an example showing that the $\mathcal{C}^1$ analog of Theorems \ref{homeotoalgebraic} or  \ref{homeotopolynomial} is false in the real case (this example is well-known, see \cite{bochnakkucharz1984} for example).  
The germ $(V,0)\subset (\R^3,0)$, defined by the vanishing of
$$f(t,x,y)= xy(y-x)(y-(3+t)x)(y-\gamma(t)x)$$
where $\gamma(t)\in\R\{t\}$ is transcendental and $\gamma(0)=4$, is not $C^1$-diffeomorphic to 
the  germ of an algebraic set as follows from the argument of  Whitney, cf. Section 14 of \cite{whitney}.   Indeed $V$ is the union of five smooth surfaces intersecting along the $t$-axis and its tangent cone at the point $(t,0,0)$ is the union of five planes intersecting along a line. The cross-ratio of the first four planes is $3+t$ and the cross-ratio of the first three and the last plane is $\gamma(t)$. Since the cross-ratio is preserved by linear maps, these two cross-ratios are preserved by $C^1$-diffeomorphisms. But these two cross-ratios are algebraically independent thus the  image of $V$ under a $C^1$-diffeomorphisms cannot be algebraic. 
\end{example}

\begin{example}\label{whitney2}
The previous example also shows that the $C^1$ analogs of Theorems \ref{homeotoalgebraic} or  \ref{homeotopolynomial} are false in the complex case.
Define $V$ in a neighborhood of $0$ in $\C^3$ by the vanishing of the polynomial of Example \ref{whitney1}.
Modifying Whitney's argument (cf.  \cite{whitney}, pp 240, 241) we will show that the germ of $V$ 
at $0$ is not $C^1$-equivalent to any Nash germ in $\mathbb{C}^3.$

For any $(t,0,0)\in\mathbb{C}^3$ with $|t|$ small, the tangent cone to $V$ at $(t,0,0)$ is
the union of five two-dimensional $\mathbb{C}$-linear spaces $L_{1,t},\ldots,L_{5,t},$ where
$L_{j,t}$ corresponds to the $j$'th factor of $f.$
Suppose that there is a $C^1$-diffeomorphism
$\Phi:(\mathbb{C}^3,0)\rightarrow(\mathbb{C}^3,0)$ such that $\Phi(V)$ is a germ of a Nash
set in $\mathbb{C}^3.$ Then the tangent cone to $\Phi(V)$ at $\Phi{(t,0,0)}$ is the union of
$d_{(t,0,0)}\Phi(L_{j,t}),$ for $j=1$,..., 5.  In particular, every
$d_{(t,0,0)}\Phi(L_{j,t})$ is a $\mathbb{C}$-linear subspace of $\mathbb{C}^3$ of dimension
$2.$

Let us check that for every $t\in\mathbb{R}$ with $|t|$ small and for every pairwise distinct
$k_1,\ldots,k_4\in\{1,\ldots,5\},$ the cross-ratio of $L_{k_1,t}$,..., $L_{k_4,t}$ is equal to the
cross-ratio of $d_{(t,0,0)}\Phi(L_{k_1,t})$,..., $d_{(t,0,0)}\Phi(L_{k_4,t}).$ By a real line
in $\mathbb{C}^3$ we mean a set of the form $\{a+tb:t\in\mathbb{R}\}$ where
$a,b\in\mathbb{C}^3.$ For $t\in\mathbb{R},$ $L_{k_1,t},\ldots,L_{k_4,t}$
are defined by real equations so there is a real line $l_t\subset\mathbb{C}^3$ intersecting
$L_{k_1,t}\cup\ldots\cup L_{k_4,t}$ at exactly four points, say $a_{1,t}$,..., $a_{4,t}.$ 
Then the cross-ratio of
$L_{k_1,t},\ldots,L_{k_4,t}$ equals the cross-ratio of $a_{1,t}$,..., $a_{4,t}.$ Moreover,
$d_{(t,0,0)}\Phi(l_t)$ is also a real line and it intersects
$d_{(t,0,0)}\Phi(L_{k_1,t})\cup\ldots\cup d_{(t,0,0)}\Phi(L_{k_4,t})$ at
$d_{(t,0,0)}\Phi(a_{1,t})$,..., $d_{(t,0,0)}\Phi(a_{4,t}).$ The cross-ratio of the last four points
equals that of $a_{1,t}$,..., $a_{4,t}$ because $d_{(t,0,0)}\Phi$ is $\mathbb{R}$-linear and
$a_{1,t},\ldots,a_{4,t}\in l_t.$  
Since the cross-ratio of $d_{(t,0,0)}\Phi(a_{1,t})$,...,  $d_{(t,0,0)}\Phi(a_{4,t})$ equals the cross-ratio of
$d_{(t,0,0)}\Phi(L_{k_1,t})$,..., $d_{(t,0,0)}\Phi(L_{k_4,t}),$ we obtain our claim.

Now observe that the complex $t$-axis is the singular locus of $V$ and its image $S$ by
$\Phi$ is the singular locus of $\Phi(V).$ Clearly, $S$ is a smooth complex Nash curve.
Moreover, the cross-ratios $h_1, h_2$ of
$d_{(t,0,0)}\Phi(L_{1,t}),\ldots,d_{(t,0,0)}\Phi(L_{4,t})$ and $d_{(t,0,0)}\Phi(L_{1,t})$,
$d_{(t,0,0)}\Phi(L_{2,t})$, $d_{(t,0,0)}\Phi(L_{3,t})$, $d_{(t,0,0)}\Phi(L_{5,t}),$ respectively,
de\-pend algebraically on $s=\Phi(t,0,0)\in S$ (cf. \cite{whitney}, p 241), i.e. $h_1,
h_2:S\rightarrow\mathbb{C}$ are complex Nash functions. On the other hand, 
the cross-ratios of $L_{1,t},\ldots, L_{4,t}$ and of $L_{1,t}, L_{2,t},
L_{3,t}, L_{5,t}$ equal $3+t$ and $\gamma(t),$ respectively.

The last two paragraphs imply that for $t\in\mathbb{R}$ with $|t|$ small, we have
$h_1(\Phi(t,0,0))=3+t$ and $h_2(\Phi(t,0,0))=\gamma(t).$ Since $S$ is a smooth complex Nash curve,
we may assume that $h_1,h_2$ are defined in some neighborhood of $0\in\mathbb{C}$ and that
$\Psi(t)=\Phi(t,0,0)$ is a map into $\mathbb{C}.$ We have $\Psi(t)=\Psi_1(t)+i\Psi_2(t)$ where
$\Psi_1,\Psi_2$ are real valued continuous functions and $h_1(s)=u_1(s)+iv_1(s),$
where $u_1,v_1$ are real Nash functions, and
$u_1(\Psi_1(t),\Psi_2(t))=3+t,$ and $v_1(\Psi_1(t),\Psi_2(t))=0$ for $t\in\mathbb{R}$ with
$|t|$ small. Since $u_1, v_1$ satisfy the Cauchy-Riemann equations and $h_1$
is not constant, neither of $u_1, v_1$ is constant. Consequently, $\Psi_1|_{\mathbb{R}},$
$\Psi_2|_{\mathbb{R}}$ are semi-algebraic functions, which contradicts the fact that
$h_2(\Psi(t))=\gamma(t)$ for real $t.$
\end{example}

\begin{example}\label{2functions1}
Theorem \ref{homeotopolynomial} cannot  be extended to many functions or to maps to $\K^m$, $m>1$. For example the one variable  analytic germs $x$ and $e^{x}-1$ cannot be made polynomial (or Nash) simultaneously  by composing with the same homeomorphism. 
\end{example}

\begin{example}\label{2functions2}
The key point in the previous examples is the fact that two one variable functions which are algebraically independent remain algebraically independent after composition with a homeomorphism.  
Theorem \ref{homeotopolynomial} also cannot  be extended to many functions, even if we assume them algebraically depended.  For instance, one variable Nash germs $x$ and $y(x) =\sqrt{\varphi (x) } -2 $, with
$\varphi (x) =  (x-1)(x+2)(x-2) $, cannot be made polynomial simultaneously since the cubic $y^2 =\varphi(x) $ is not rational.  

\end{example}


\end{document}